\documentclass{article}

\usepackage[T1]{fontenc}
\usepackage[english]{babel}

\usepackage[a4paper,top=2cm,bottom=2cm,left=2cm,right=2cm,marginparwidth=1.75cm]{geometry}

\usepackage{amsmath}
\usepackage{bbold}
\usepackage{dsfont}
\usepackage{graphicx}
\usepackage[colorlinks=true, allcolors=blue]{hyperref}
\usepackage{amsthm}
\usepackage{wrapfig}
\usepackage{amssymb}
\usepackage{stmaryrd}
\usepackage{csquotes}

\newtheorem{prop}{Proposition}[subsection]
\newtheorem{theorem}[prop]{Theorem}
\newtheorem{lemma}[prop]{Lemma}
\newtheorem{definition}{Definition}[subsection]
\newtheorem*{remark}{Remark}

\DeclareMathOperator{\Deg}{Deg}

\newcommand{\Mod}{\,\mathrm{mod}\,}
\newcommand{\e}{\mathrm{e}}
\newcommand{\de}{\,\mathrm{d}}
\newcommand{\ii}{\mathrm{i}}
\newcommand{\Po}{\mathcal{P}}

\newcommand{\Q}{\mathds{Q}}
\newcommand{\I}{\mathcal{I}}
\newcommand{\Ev}{\mathcal{E}}
\newcommand{\Pro}{\mathds{P}}
\newcommand{\E}{\mathds{E}}
\newcommand{\F}{\mathds{F}}
\newcommand{\Z}{\mathds{Z}}
\newcommand{\N}{\mathds{N}}
\newcommand{\R}{\mathds{R}}
\newcommand{\C}{\mathds{C}}

\newcommand{\Indic}{\mathbb{1}}
\newcommand{\eps}{\varepsilon}
\newcommand{\bs}{\mathbf}

\title{Irreducibility of random polynomials of $\Z[X]$}
\author{Pierre-Alexandre Bazin\footnote{Universit\'e Paris Cit\'e, Sorbonne Université, CNRS \\ Institut de Math\'ematiques de Jussieu-Paris Rive Gauche, 75013 Paris, France. \\ Email: bazin@imj-prg.fr}}
\date{}

\begin{document}
\maketitle

\begin{abstract}
    In \cite{koukoulo}, Bary-Soroker, Koukoulopoulos and Kozma proved that when $A$ is a random monic polynomial of $\Z[X]$ of deterministic degree $n$ with coefficients $a_j$ drawn independently according to measures $\mu_j,$ then $A$ is irreducible with probability tending to $1$ as $n\to\infty$ under a condition of near-uniformity of the $\mu_j$ modulo four primes (which notably happens when the $\mu_j$ are uniform over a segment of $\Z$ of length $N\ge 35.$) We improve here the speed of convergence when we have the near-uniformity condition modulo more primes. Notably, we obtain the optimal bound $$\Pro(A \text{ irreducible}) = 1 -O(1/\sqrt n)$$ when we have near-uniformity modulo twelve primes, which notably happens when the $\mu_j$ are uniform over a segment of length $N\ge 10^8.$
\end{abstract}


\maketitle

\section{Introduction}
\subsection{Main results}

We consider $A$ a random monic polynomial in $\Z[X]$ of determinisitic degree $n$ with independently chosen coefficients. Formally, we introduce measures $(\mu_j)_{0\le j<n}$ on $\Z$ (generally taken all equal to the same measure~$\mu$), and we define our random monic polynomial $$A(X) = \sum_{j=0}^n a_jX^j \in \Z[X],$$ with $a_n = 1, a_j \sim \mu_j$ for all $j < n$ and the $a_j$ are independent.

When the $\mu_j$ are all taken equal to the same $\mu$ uniform over $\{0,1\},$ the problem of the irreduciblity of $A$ when $n\to\infty$ has been first studied by Konyagin \cite{konyagin}. The latest results, given below, have been achieved by Bary-Soroker, Koukoulopoulos and Kozma \cite{koukoulo}. 

\begin{theorem}[Bary-Soroker, Koukoulopoulos and Kozma \cite{koukoulo}]\label{kk}

Assume the $\mu_j$ are all the same and uniform over a segment $\llbracket a, a+N-1\rrbracket$ of length $N\ge 2.$ Then there are absolute constants $c,K,\delta >0$ such that when $n\ge(\log(a+N))^3,$ $$\Pro(A \text{ has a factor of degree}\le cn) \le \Pro(a_0 = 0) + Kn^{-\delta}.$$

When $N\ge 35,$ we further get $$\Pro(A \text{ reducible}) \le \Pro(a_0 = 0) + Kn^{-\delta}.$$
\end{theorem}
\begin{remark}\label{zero}
    The $\Pro(a_0 =0)$ term corresponds to the probability that $X$ divides $A,$ which does not depend on $n$ (and is equal to $0$ or $1/N$ depending on $a$).
\end{remark}

Building on the ideas from this paper, we achieve a better convergence speed, namely

\begin{theorem}
\label{uniform-case}

Assume the $\mu_j$ are all the same and uniform over a segment $\llbracket a, a+N-1\rrbracket$ of length $N\ge 2.$ Then there are absolute constants $c,K >0$ such that when $n\ge (\log (|a|+N))^3,$ $$\Pro(A \text{ has a factor of degree}\le cn) \le \Pro(a_0 = 0) + K/\sqrt n.$$

When $N\ge 10^8,$ we further get $$\Pro(A \text{ reducible}) \le \Pro(a_0 = 0) + K/\sqrt n.$$
\end{theorem}

The $O(1/\sqrt n)$ bound here is optimal: indeed, $X+1$ divides $A$ with probability $\Pro(A(-1)=0),$ which decreases as $1/\sqrt n$ when $N$ is fixed. It can thus be useful to treat small cyclotomic factors (which, like $X+1,$ have high probability of appearing) separately in order to obtain better convergence speeds. Indeed, Breuillard and Varj\'u \cite{bvarju} proved the following result assuming the Generalised Riemann Hypothesis (GRH) 
(their result is given for general measure $\mu,$ but we specify it here for $\mu$ uniform over a segment).
\begin{theorem}[Breuillard and Varj\'u \cite{bvarju}, Corollary 3]\label{bvarj}
    Assume GRH and assume the $\mu_j$ are all equal to the same measure $\mu$ of finite second moment, not supported on a singleton. Fix $C > 0$ and let $U_C$ the finite subset of $\C$ consisting of $0$ and all roots of unity $\omega$ such that $[\Q(\omega):\Q] < C.$ Then for all $n,$ $$\Pro(A \text{ reducible}) = \Pro\big(\exists \omega\in U_C : A(\omega) = 0\big) + O_{\mu,C}(n^{-C/2}).$$
\end{theorem}

We used here the usual big-Oh notation: we write "$f(n) = O_C(g(n))$ when $n\ge n_0$" if for some constant $K$ depending only on $C,$ $|f(n)|\le K|g(n)|$ for all $n\ge n_0.$ We will also note $f(n) = O(n^{C+o(1)})$ when $f(n) = O_{\delta}(n^{C+\delta})$ for all $\delta>0.$

Without the Riemann hypothesis, we achieve here the same precision as Theorem \ref{bvarj} when $\mu$ is uniform over a segment large enough compared to $C.$

\begin{theorem}
\label{main-uniform}

Assume the $\mu_j$ are all the same and uniform over a segment $\llbracket a, a+N-1\rrbracket$ of length $N\ge 2.$ Fix $C > 0$ and let $U_C$ the finite subset of $\C$ consisting of $0$ and all roots of unity $\omega$ such that $[\Q(\omega):\Q] < C.$ We then have constants $\delta$ and $N_0$ depending only on $C$ such that when $n\ge (\log (|a|+N))^3,$
\begin{itemize}
    \item when $N\ge N_0,$ $\Pro(A \text{ reducible}) = \Pro\big(\exists \omega\in U_C : A(\omega) = 0\big) + O_{C}(n^{-C/2});$
    \item in the general case $N\ge 2,$ we only get $$\Pro(A \text{ has a factor of degree }\le \delta n) = \Pro\big(\exists \omega\in U_C : A(\omega) = 0\big) + O_{C}(n^{-C/2}).$$
\end{itemize}
\end{theorem}

When $C = 0.01,$ we can take $N_0 = 35$ which gives back Theorem \ref{kk} (and decreasing $C$ does not decrease $N_0$ further); when $C = 1,$ we can take $N_0 =10^8.$ In general, values of $N_0$ are given by equation~(\ref{bound-N}). Theorem \ref{uniform-case} then immediately follows from Theorem \ref{main-uniform} with $C = 1.$

\subsection{A master theorem}

We give the main result with general measures $\mu_j,$ from which we will derive Theorem \ref{main-uniform}. We first recall a definition

\begin{definition}
For a measure $\mu$ on $\Z,$ we note $$\widehat\mu : \theta\in\R/\Z \longmapsto \E_\mu\left(\e^{2\ii\pi\theta Y}\right) = \sum_{j \in \Z} \mu(j)\e^{2\ii\pi\theta j}$$ its Fourier transform.
\end{definition}
\begin{theorem}
\label{general-case}
    Let $A$ a random monic polynomial of degree $n$ with independent coefficients following probability measures $\mu_j$ and $P = p_1\ldots p_r$ a product of $r\ge 4$ distinct primes.
    
    Assume we can find an integer $s\ge 1$ (independent of $n$) such that for all $j$ and $Q,R,\ell$ with $QR = P,\ Q > 1$ and $\ell \in \Z/R\Z,$ we have \begin{equation}\label{condition-hatmu}
        \sum_{k\in \Z/Q\Z} \big|\widehat{\mu_j}(k/Q+\ell/R)\big|^s\le \left(1-\frac1{\log n}\right)\sqrt Q.
    \end{equation} 
    Then when $n\ge P^4$ and $n_1\le n/2s,$ we have $$\Pro\big(\exists D|A \text{ of degree } \in [n_1, n/2s]\big) = O_{r,s}\left(n_1^{-rQ\left(\frac{1-1/r}{\log 2}\right) + o(1)}\right),$$ with $Q(t) := t\log t - t +1 = \int_1^t \log u\,\de u.$
\end{theorem}

Bary-Soroker, Koukoulopoulos and Kozma \cite[Theorem 7]{koukoulo} obtained the same result with a bound~$O_s(n^{-c}) ;$ the key difference is that we are able to achieve an upper bound that scales with the number~$r$ of primes that are considered.

Here, the condition (\ref{condition-hatmu}) on the $\widehat{\mu_j}$ translates to a kind of near-uniformity modulo $P$ : indeed, the sum is minimised when $\mu$ is uniform mod $P,$ in which case we have $|\widehat{\mu_j}(a/P)| = 0$ for all $a\not\equiv 0\mod P.$ We will show in Section \ref{uniform-final} that (\ref{condition-hatmu}) is verified with $s=1$ when the $\mu_j$ are uniform over a segment of length $N$ large enough in terms of $P.$ 

To prove Theorems \ref{uniform-case} and \ref{main-uniform} though, Theorem \ref{general-case} is not sufficient: indeed, it only gives a useful bound when $n_1$ grows to infinity with $n.$ Thus, we need to exclude the possibility of small degree factors separately. These have already been treated by Konyagin \cite{konyagin} for polynomials with coefficients in $\{0,1\}.$ In the more general setting of Theorem \ref{main-uniform}, we can use the result proved in \cite[§7]{koukoulo}.

\begin{prop}[\cite{koukoulo}, §7]\label{small-degree-lemma}
    Suppose all the $\mu_j$ have support in $[-H,H]$ for some $H>0,$ and do not take any value with probability higher than $1-c$ for some $c>0.$ Then for $n\ge(\log H)^3,$ $$\Pro\big(A \text{ has a non-$X,$ non-cyclotomic irreducible factor of degree }\le n^{1/10}\big) = O_c\left(\e^{-\sqrt n}\right).$$
\end{prop}
\begin{remark}
    This result is not directly stated anywhere in \cite{koukoulo}, but appears in the proof of \cite[proposition~2.1]{koukoulo}, which combines it with a weaker result on cyclotomic divisors. However, in order to achieve the better precision in Theorem \ref{main-uniform} we need a stronger result on cyclotomic divisors, given below.
\end{remark}

\begin{lemma}\label{small-cyclo}
    Assume all the $\mu_j$ do not take any value with probability higher than $1-c$ for some $c>0.$ Then for $\Phi_d$ the $d$-th cycloctomic polynomial (of degree $\varphi(d)$), we have $$\Pro(\Phi_d|A) \le \left(\frac{Kn}{cd}\right)^{\varphi(d)/2}$$ for some absolute constant $K.$
\end{lemma}
\begin{remark}
    When all $\mu_j$ are equal, this is \cite[lemma 45]{bvarju} with a more precise dependence in $\mu.$ The proof for general $\mu_j$ is given in Section \ref{small-degree}.
\end{remark}

Using Theorem \ref{general-case} with $n_1 = n^{1/10}$ together with Proposition \ref{small-degree-lemma} is sufficient to prove Theorem \ref{main-uniform}, but does not give the best values of $N_0(C)$ (and, in particular, would only achieve a $O(n_1^{-1/2}) = O(n^{-1/20})$ error term when $\mu$ is uniform over a segment of length $N = 10^8,$ instead of $O(n^{-1/2})$). The trick in order to achieve the better values of $N_0$ is to first use Theorem \ref{general-case} with $n_1 = n^{1/10}$ and large values of $r$ and $s$ in order to rule out factors in $[n^{1/10}, \eps n]$ for some $\eps > 0,$ and then use Theorem \ref{general-case} again, this time with a smaller value of $r$ (for which we can take $s=1$) together with $n_1 = \eps n$ in order to rule out factors in $[\eps n, n/2].$ The detailed argument is given in Section \ref{uniform-final}. 

\subsection{Definitions and notations}
\label{definitions}

We now give definitions that will be useful for the proof of Theorem \ref{general-case}. When $p$ is a prime, we note $\F_p[X]^u$ the set of monic polynomials of $\F_p[X],$ and $\F_p[X]_n^u$ the set of monic polynomials of $\F_p[X]$ of degree $n.$ We now fix $p_1,\ldots, p_r$ distinct primes and define $$\Po := \prod_{i=1}^r \F_{p_i}[X]^u$$  the set of $r$-uples of monic polynomials each in $\F_{p_i}[X].$ We then define multiplication and divisibility in $\Po$ pointwise (that is, $\bs D = (D_1,\ldots, D_r)$ divides $(A_1,\ldots, A_r)$ if and only if each $D_i$ divides $A_i$ in $\F_{p_i}[X]$). As usual, we use $\bs{bold\ characters}$ to denote $r$-uples.

\begin{remark}
    Here, we restricted ourselves to monic polynomials, as we are only interested in the multiplicative structure of $\Po.$
\end{remark}

For $\bs A = (A_1,\ldots, A_r) \in \Po,$ we then define its degree $\deg \bs A := (\deg A_1,\ldots, \deg A_r),$ and its total degree $$\Deg \bs A := \sum_{i = 1}^r \deg A_i$$ as well as its norm $$\|\bs A\| := \prod_{i = 1}^r p_i^{\deg A_i}.$$ 

For a fixed degree $\bs d = (d_1,\ldots, d_r),$ we then define $$\Po_{\bs d} := \prod_{i=1}^r \F_{p_i}[X]_{d_i}^u$$ the set of degree-$\bs d$ elements in $\Po.$ Notice all elements in $\Po_{\bs d}$ have the same norm $\prod_i p_i^{d_i},$ which is also the cardinality of $\Po_{\bs d}.$ 

Here, any element of $(\Po,\cdot)$ can be uniquely factorised as a product of irreducibles, where an irreducible of $\Po$ is an element of the form $(1,\ldots,1,I_i,1,\ldots,1)$ with $I_i$ irreducible in $\F_{p_i}[X]$ (indeed, the factorisation is obtained as the product of factorisations of each component). For $\bs I$ an irreducible of $\Po$ and $\bs A\in\Po,$ we note $\nu_\bs I(\bs A)$ the number of times $\bs I$ appears as a factor in $\bs A.$ We also note $$\tau(\bs A) := \prod_{\bs I\text{ irreducible}} (\nu_{\bs I}(\bs A)+1)$$ the number of divisors of $\bs A$ in $\Po.$

We denote by $$\bs X_i = (1,\ldots,1, X \in \F_{p_i}[X],1,\ldots,1)$$ the irreducible obtained for $I_i = X$ above. In particular, for any $i,$ we have $\Deg \bs X_i = 1, \|\bs X_i\| = p_i,$ and $\bs X_i$ divides $\bs B\in \Po$ if and only if $X$ divides $B_i$ in $\F_{p_i}[X].$ We note $\Po^X$ (resp. $\Po_\bs d^X$) the set of elements of $\Po$ (resp. $\Po_\bs d$) that are not divisible by any $\bs X_i.$

We give now a few more definitions. 

\begin{definition}
\label{def-delta-P}
    Given $\bs A = (A_1,\ldots, A_r)$ a random variable in $\Po,$ we define the discrepancy to uniformity modulo polynomials of degree $\le m$ of $\bs A$ by \begin{equation}\label{def-Delta}
        \Delta_\bs A(m) := \sum_{\substack{\bs B \in \Po^X \\ \max\deg \bs B\le m}} \left|\Pro(\bs B|\bs A) - \frac1{\|\bs B\|}\right|.
    \end{equation}

    Here, $\max\deg\bs B$ is the maximum of the $\deg B_i$ for $1\le i\le r.$
\end{definition}

\begin{remark}
    We need to exclude here the factors $\bs B$ multiples of some $\bs X_i$ from the sum, as they would otherwise contribute too much when the $A_i$ are taken equal to $A\Mod p_i$ where $A$ is the random polynomial from Theorem~\ref{general-case}.
\end{remark}

\begin{definition}
\label{m-friable}
    For a fixed degree $m\in\N,$ we define $\I_m$ the set of irreducibles of $\Po$ of degree $\le m$ and not of the form $\bs X_i,$ and $$\Sigma_m := \sum_{\bs I\in \I_m} \frac1{\|\bs I\|}\,,$$ $$\Pi_m := \prod_{\bs I\in \I_m}\left(1-\frac1{\|\bs I\|}\right).$$

    For $\bs A \in \Po,$ we then define its $m$-friable part $$\bs A_{\le m} = \prod_{\bs I \in \I_m} \bs I^{\nu_\bs I(\bs A)},$$ and its $m$-sieved part $$\bs A_{> m} = \bs A/\bs A_{\le m} = \prod_{\bs I \not\in \I_m} \bs I^{\nu_\bs I(\bs A)}.$$ We say $\bs A$ is $m$-friable if $\bs A = \bs A_{\le m}$ (that is, all its irreducible factors are in $\I_m$).

    Finally, when $\bs A$ is a random variable in $\Po,$ we define the event $\Ev_m$ "$\bs A$ does not have many factors of low degree" by \begin{equation}\label{def-Ev-m}
        \Ev_m := \left\{\Deg \bs A_{\le m}\le m(\Sigma_m-2), \tau(\bs A_{\le m})\le \e^{(1-1/r)\Sigma_m}\right\}.
    \end{equation}
\end{definition}
\begin{remark}
    Here, we needed to exclude the $\bs X_i$ factors from the $m$-friable part as our bound on $\Delta_\bs A(m)$ gives no control over these.
\end{remark}

\subsection{Outline of the proof of Theorem \ref{general-case}}

Our proof follows the ideas of Bary-Soroker, Koukoulopoulos and Kozma \cite{koukoulo}. If $D$ is a factor of $A$ of degree~$k,$ this will also be true modulo any prime $p.$ Thus, when we fix $p_1,\ldots, p_r$ distinct primes and define $$\bs A = A_P := (A\Mod p_i)_{1\le i\le r} \in \Po_{(n,\ldots, n)},$$ if $A$ has a factor $D$ of degree in $[n_1,n_2],$ then $\bs A$ has a factor $\bs D = D_P$ of degree $(k,\ldots, k)$ for some $k\in [n_1,n_2].$

We are thus interested in the presence of factors of degree $(k,\ldots, k)$ of a random variable in $\Po.$ It is well known that the multiplicative structure of $\F_p[X]^u$ is analogous to that of integers, and $\Po$ is a concatenation of $r$ spaces $\F_{p_i}[X]^u.$ In the case of integers, Ford \cite{ford} achieved a precise estimate on the presence of factors in an interval $[2^k, 2^{k+1}]$ (which corresponds to degree $k$ in the polynomial-integer analogy) :

\begin{theorem}[Ford, \cite{ford}]
    If $a$ is a random integer taken uniformly in $[1,x]$ with $x\ge 2^{2k},$ then $$\Pro\big(a \text{ has a factor in }[2^k, 2^{k+1}]\big) \asymp k^{-Q(1/\log 2)}(\log k)^{-3/2}.$$
\end{theorem}

As we are only interested in a precision up to a $k^{o(1)}$ term in Theorem \ref{general-case}, we will not need such a precise result in $\Po.$ Still, this tells us we can expect a $O(k^{-rQ(1/\log 2) + o(1)})$ bound on $\Pro(\bs A \text{ has a factor of degree }(k,\ldots, k))$ when $\bs A$ is uniform in $\Po_{(n,\ldots, n)}$ with $n\ge 2k,$ which happens for $\bs A=A_P$ when the $\mu_j$ are uniform modulo $P = p_1\ldots p_r.$ Here, we got an extra $r$ in the exponent of the bound as we want a factor modulo $r$ different primes. When $\bs A$ is not exactly uniform in $\Po_{(n,\ldots,n)},$ we will get an extra error term controlled by some $\Delta_\bs A(k+o(k)).$

Summing on $k$, we can then expect a $O\left(n_1^{1-rQ(1/\log 2)}\right)$ bound on Theorem \ref{general-case} when $r\ge \left\lceil \frac{1}{Q(1/\log 2)}\right\rceil = 12.$ However, it turns out we do not need that many primes. Indeed, \cite{bkozma} and then \cite{koukoulo} achieved a result with only $r = 4$ primes. We first explain the argument of Bary-Soroker, Koukoulopoulos and Kozma that allowed them to achieve this lower value of $r,$ and then explain how we modify it to obtain better bounds for large values of $r.$

Pick $p$ a prime, $m\le k \le n/2$ and $A$ uniformly distributed in $\F_p[X]_n^u.$ Then, conditionally on its $m$-friable part $A_{\le m},$ the expected number of factors of degree $k$ of $A$ is about $m^{-1} \tau(A_{\le m}).$ The underlying reason for the proportionality with $\tau(A_{\le m})$ is that when $D$ is a factor of $A,$ there are $\tau(A_{\le m})$ choices for $D_{\le m}$ (the contribution of the case $\deg A_{\le m} > k$ is negligible when $m$ is small enough - taking $m \ll k/\log k$ will be sufficient). 

Now, note that $\tau(A_{\le m})$ behaves roughly like $2^W$ where $W$ follows a Poisson law of parameter $\log m$ (which corresponds to the behaviour of $\omega(A_{\le m})$). In particular, while the average value of $\tau(A_{\le m})$ is around $m,$ its typical value is around $2^{\log m}.$ This means, while the average number of factors of degree $k$ is one, there are very few polynomials (those with $\tau(A_{\le m}) \gg m$) which contribute mostly to that average, and the others are unlikely to have factors of degree $k,$ for all $k \gg m^{1 + \eps}.$ In particular there is a significant overlap between the events "$A$ has a factor of degree $k$" for all large $k.$

Bary-Soroker, Koukoulopoulos and Kozma \cite{koukoulo} then considered $\bs A = (A_1,\ldots, A_r) \in\Po$ with $r \ge 4 > \frac{1}{1 - 1/\log 2}$ and a dyadic interval $[v, 2v]$ with $v \asymp m^{1 + \eps},$ and considered two cases.

\begin{itemize}
    \item Either one of the $A_j$ has more than $(\log m)\frac{1-1/r-\eps}{\log 2}$ irreducible factors of degree $\le m.$ This happens with probability $\ll m^{-2\eps} \ll v^{-\eps}$ when $\eps$ is small enough since $1 - 1/r > \log 2.$
    \item Or all of the $A_j$ have less than $(\log m)\frac{1-1/r-\eps}{\log 2}$ irreducible factors of degree $\le m.$ In that case, for all $k$ in $[v,2v],$ the expected number of factors of factors of $A_j$ of degree $k$ is $\le m^{-1} 2^{(\log m)\frac{1-1/r-\eps}{\log 2}} = m^{-1/r - \eps}$ (the case where $(A_j)_{\le m}$ is not squarefree does not significantly affect the computation). In particular, assuming all $A_j$ are independent, for fixed $k$ the probability that all $A_j$ have a factor of degree $k$ is $\le m^{-1-r\eps}.$ Summing on $k \in [v, 2v]$ and since $v\asymp m^{1+\eps},$ we get that $\bs A$ has a factor of degree $(k,\ldots, k)$ for some $k$ in $[v, 2v]$ with probability less than $v^{-\eps}$ in this case.
\end{itemize}

This proves $\Pro(\exists k\in [v, 2v], \bs A \text{ has a factor of degree }(k,\ldots, k)) \ll v^{-\eps}$ when $r\ge 4,$ which can then be summed on $v.$ We now want to make the heuristic argument above more precise in order to achieve a bound  $O\left(v^{-rQ\left(\frac{1-1/r}{\log 2}\right) + o(1)}\right)$ instead of $O(v^{-\eps}).$

First, we note that when all $A_j$ are independent, the expected number of factors of degree $(k,\ldots, k)$ of $\bs A$ conditionally on $\bs A_{\le m}$ is roughly $$\prod_{i=1}^r \E\left(\exists D_j|A_j \text{ of degree } k \,\Big|\, (A_j)_{\le m}\right) \approx m^{-r} \tau(\bs A_{\le m})$$ (the error terms will be precised later when proving Proposition \ref{main-mod-pi}), so that in order to bound the probability of the existence of factors of degree $(k,\ldots,k)$ we in fact only need a bound on the average $\tau(\bs A_{\le m})$ rather than on all the $\tau\big((A_j)_{\le m}\big).$ We now consider a range $[v,2v]$ with $v = m^{1+o(1)},$ and estimate 

\begin{align*}
    \Pro\left(\exists \bs D\in \bigcup_{v\le k\le 2v} \Po_{(k,\ldots, k)},\, \bs D|\bs A\right) &= \E\left(\Pro\left(\exists \bs D\in \bigcup_{v\le k\le 2v} \Po_{(k,\ldots, k)},\, \bs D|\bs A \,\Big|\, \bs A_{\le m}\right)\right)
    \\&\le \E\left(\min\left\{1;\sum_{v\le k\le 2v} \E\left(\exists \bs D\in \Po_{(k,\ldots, k)},\, \bs D|\bs A\,\Big|\, \bs A_{\le m}\right)\right\}\right)
    \\&\approx \E\left(\min\left\{1;vm^{-r}\tau(\bs A_{\le m})\right\}\right)
    \\&= v^{-rQ\left(\frac{1-1/r}{\log 2}\right) + o(1)}
\end{align*}
when $v = m^{1+o(1)}$ and $\tau(\bs A_{\le m})$ behaves like $2^W$ with $W$ following a Poisson law of parameter $r\log m.$

For a general random variable $\bs A$ in $\Po,$ we prove in Section \ref{large-degree-k} the following result. This improves \cite[proposition 2.2]{koukoulo} by giving a bound that improves with the number $r$ of primes considered.

\begin{prop}\label{main-mod-pi}
    
    Let $\bs A$ a random variable in $\Po$ such that $\Pro(\bs X_i^\nu|\bs A) \le (1-1/L)^\nu$ for some $L$ and for all $i,\nu.$  Then for $n_1\le n_2,$ $$\Pro\left(\exists \bs D\in \bigcup_{n_1\le k\le n_2} \Po_{(k,\ldots, k)},\, \bs D|\bs A\right) = O_r\left(L^r n_1^{-rQ\left(\frac{1-1/r}{\log 2}\right) + o(1)} + n_2^{5r}\Delta_\bs A\left(n_2 + n_2/(\log n_2)^3\right)\right).$$
\end{prop}

In Proposition \ref{main-mod-pi}, we have an error term controlled by some $\Delta_\bs A(n_2 + o(n_2))$ (which does not control the factors that are multiples of the $\bs X_i$), and a separate condition on the $\Pro(\bs X_i^\nu|\bs A).$ This is so that Proposition \ref{main-mod-pi} can be applied when $\bs A = A_P = (A \Mod p_i)_i$ with $A$ the random polynomial from Theorem \ref{general-case}. To bound the $\Delta_\bs A$ error term, we use the results obtained by Bary-Soroker, Koukoulopoulos and Kozma in \cite[proposition 2.3]{koukoulo}, given below.

\begin{prop}[\cite{koukoulo}, proposition 2.3]\label{bound-Delta}

     Let $A$ a random monic polynomial of degree $n$ with independent coefficients following probability measures $\mu_j$ and $P = p_1\ldots p_r$ a product of $r$ distinct primes, with $P\le n^{1/4}.$
    
    If $s\le\log n$ and $\gamma\ge 1/2$ are such that for all $j$ and $Q,R,\ell$ with $QR = P,\ Q > 1$ and $\ell \in \Z/R\Z,$ we have $$\sum_{k\in \Z/Q\Z} |\widehat{\mu_j}(k/Q+\ell/R)|^s\le \left(1-\frac1{\log n}\right)Q^{1-\gamma},$$ then with $\bs A := (A\Mod p_i)_{1\le i\le r},$ we have $$\Delta_{\bs A}\left(\gamma n/s + n/(\log n)^3\right) \le \e^{-n^{1/10}}.$$
\end{prop}

\begin{remark}
    The proposition 2.3 as written in \cite{koukoulo} only bounds this error term for degree up to $\gamma n/s + n^{0.88};$ however, to obtain the constant $rQ\left(\frac{1-1/r}{\log 2}\right)$ in Theorem \ref{general-case}, we need to go up to $\gamma n/s + n^{1-o(1)}.$ To cover these higher degrees, notice the relevant bound at the end of the proof of \cite[proposition 2.3]{koukoulo} to cover degree $L = \gamma n/s + \eps(n) > \gamma n/s$ is to have $P^{\eps(n)}\alpha^{n/s}$ small enough, where $\alpha := 1-1/\log n$ is the term in front of $Q^{1-\gamma}$ in the condition on the $\widehat{\mu_j}$ (which was taken equal to $1-n^{-1/10}$ in \cite{koukoulo}).

    Here when $\eps(n)\le n/(\log n)^3$ and $\alpha = 1-1/\log n,$ we indeed have as $P\le n^{1/4}$ and $s\le\log n,$ $$P^{\eps(n)}\alpha^{n/s} \le \e^{\frac{n\log P}{(\log n)^3} -\frac{n}{s\log n}} \le \e^{-\frac34\frac{n}{(\log n)^2}},$$ hence our version of the proposition holds.
\end{remark}

We can now prove Theorem \ref{general-case} combining Propositions \ref{main-mod-pi} and \ref{bound-Delta}.

\begin{proof}[Proof of Theorem \ref{general-case}]

    Here, we can assume $n$ is large enough compared to $r$ and $s,$ as the result for the smaller values of $n$ can be made trivial adjusting the implicit constant of the big-Oh term. As we also assumed $n\ge P^4,$ Proposition \ref{bound-Delta} applies with $\gamma = 1/2$ and gives us with $n_2 := n/2s,$
    $$\Delta_\bs A\left(n_2 + n_2/(\log n_2)^3\right)\le \Delta_\bs A\left(n/2s + n/(\log n)^3\right) \le \e^{-n^{1/10}} \le n^{-5r},$$ where $\bs A = A_P = (A\Mod p_i)_{1\le i\le r}.$

    Moreover, we have for all $i,\nu,$ if the $a_j$ are the coefficients of $A$ (drawn independently according to the~$\mu_j$), $$\Pro\left(\bs X_i^\nu|\bs A\right) = \prod_{0\le j<\nu}\Pro(a_j \equiv 0\Mod p_i) \le (1-1/4s)^\nu$$ (see \cite[relation (2.8)]{koukoulo} for this bound on the $\Pro(a_j \equiv 0\Mod p_i)$ under the condition of Theorem \ref{general-case} on the~$\widehat{\mu_i}$). Thus, we can apply Proposition \ref{main-mod-pi} with $L = 4s$ to get \begin{align*}
        \Pro\big(\exists D|A \text{ of degree } \in [n_1, n_2]\big) &\le \Pro\big(\exists \bs D|\bs A \text{ of degree }(k,\ldots,k) \text{ for some } k\in [n_1, n_2]\big) \\&= O_r\left(s^r n_1^{-rQ\left(\frac{1-1/r}{\log 2}\right) + o(1)} + n_2^{5r}\Delta_\bs A(n_2 + n_2/(\log n_2)^3)\right) \\&= O_{r,s}\left(n_1^{-rQ\left(\frac{1-1/r}{\log 2}\right) + o(1)}\right).
    \end{align*}
\end{proof}
\begin{remark}
    Given Ford's result, we expect the estimate in Proposition \ref{main-mod-pi} to be only an overestimate by a $n_1^{o(1)}$ factor when it comes to the main term. This means we cannot further reduce the number of primes $r$ (depending on $C$) required to obtain Theorem \ref{main-uniform}. Moving on to the dependency between $N_0$ and $P = p_1\ldots p_r$ given by equation (\ref{bound-N}), it likely cannot be improved by a lot either. Indeed, heuristically, when the coefficients of $A$ are drawn in a segment of length $N\le \sqrt P$ and $\bs B$ has degree $n/2$ in each component, $\bs A = A_P$ has only $N^n \le P^{n/2} = \|\bs B\|$ possible values which means we cannot estimate $\Pro(\bs B|\bs A)$ by $1/\|\bs B\|$ and thus cannot control our error term $\Delta_\bs A(n/2 + o(n)).$
\end{remark}

\section{Ruling out factors of degree $(k,\ldots, k)$ in $\Po$}\label{large-degree-k}

We show $\bs A\in\Po$ is unlikely to have a factor of degree $(k,\ldots, k)$ for some $k$ when the number of divisors of~$\bs A_{\le m}$ is low for some $m$ small compared to $k.$
\begin{prop}
\label{degree-k-prop}
    Let $m>0$ and $v \ge 6m\Sigma_m.$ Let $\bs A$ a random variable in $\Po,$ and assume there is an $L>0$ such that for all $i,\nu,$ $\Pro(\bs X_i^\nu|\bs A) \le (1-1/L)^\nu.$ Then $$\Pro\left(\left(\exists \bs D\in \bigcup_{v\le k\le 2v} \Po_{(k,\ldots, k)},\, \bs D|\bs A\right) \cap \Ev_m\right) = O_r\left(\frac vm(L\log m)^r m^{-rQ\left(\frac{1-1/r}{\log 2}\right)} + m^{5r}\Delta_\bs A(2v + 5m\Sigma_m)\right),$$
    where we recall $\Ev_m$ is defined by (\ref{def-Ev-m}), 
    $$\Ev_m := \left\{\Deg \bs A_{\le m}\le m(\Sigma_m-2), \tau(\bs A_{\le m})\le \e^{(1-1/r)\Sigma_m}\right\}.$$
\end{prop}

As the typical values of $\tau(\bs A_{\le m})$ are concentrated around $2^{\Sigma_m},$ the $\Ev_m$ are unlikely when $r\ge 4> \frac{1}{1-\log 2}$ and this will allow us to prove Proposition~\ref{main-mod-pi}. 

\subsection{Preliminary results}

We first estimate $\Sigma_m$ and $\Pi_m$ using the Prime Polynomial Theorem. 

\begin{lemma}\label{pi_m}
    We have $r(\log m -2) \le \Sigma_m \le r(\log m +1),$ and $\Pi_m \le \e^{2r}m^{-r}.$ 
\end{lemma}
\begin{proof}
    For any prime $p,$ the Prime Polynomial Theorem gives the estimate $$p^k/k - 2p^{k/2}/k \le \#\{I \in \F_p[X] \text{ monic irreducible of degree }k\} \le p^k/k$$ on the number of irreducibles of degree $k.$ We thus have
    \begin{align*}
        \Sigma_m &= \sum_{i=1}^r \sum_{k=1}^m \frac1{p_i^k}\#\{I \in \F_{p_i}[X] \text{ monic irreducible of degree }k, I\ne X\} \\ &= \sum_{i=1}^r \sum_{k=1}^m \left(\frac1k - \eps_{k,i}\right),
    \end{align*}
    with $0 \le \eps_{k,i} \le 2p_i^{-k/2}/k \le p_i^{-k/2}$ when $k\ge 2$ and $\eps_{1,i}=1/p_i \le p_i^{-1/2}$ when $k = 1$ (here we treated separately $k=1$ in order to exclude the polynomial $X$), so that $$0 \le \sum_{i=1}^r \sum_{k=1}^m \eps_{k,i} \le \sum_{i=1}^r \frac 1{\sqrt{p_i} -1} \le \frac r{\sqrt2-1}.$$ As we always have $r(\log m + \gamma) \le \sum_i\sum_k \frac1k \le r(\log m+1),$ we get $$r(\log m - 2) \le \Sigma_m \le r(\log m + 1).$$

    Then we have $$\Pi_m \le \prod_{\bs I\in \I_m} \e^{-1/\|\bs I\|} = \e^{-\Sigma_m} \le \e^{2r}m^{-r}.$$
\end{proof}

We will then need a sieve lemma (similar to \cite[lemma 8.2]{koukoulo}) to estimate $$\Pro\big(\bs D|\bs A; (\bs A/\bs D)_{\le m} = 1\big) = \Pro\big(\bs D_{>m}|\bs A; \bs A_{\le m} = \bs D_{\le m}\big)$$ with an error term depending on some $\left|\Pro(\bs B|\bs A)-\|\bs B\|^{-1}\right|$ with $\bs B$ of small enough degree (so that it can be bounded by some $\Delta_\bs A$ term).
\begin{lemma}\label{8.2}
    For any fixed $\bs D \in \Po,$ we have $$\Pro\big(\bs D|\bs A; (\bs A/\bs D)_{\le m} = 1\big) \le 2\cdot \frac1{\|\bs D\|}\Pi_m + \sum_{\substack{\bs G = \bs I_1\ldots \bs I_\ell \\ \ell\le 4\Sigma_m+2}}\left|\Pro(\bs D\bs G|\bs A)-\frac1{\|\bs D\bs G\|}\right|,$$
    where the sum is on all products $\bs G$ of $\ell\le 4\Sigma_m+2$ distinct irreducible factors of $\I_m$ (notably, $\bs G$ will always have total degree $\le m(4\Sigma_m+2).$)
\end{lemma}
\begin{proof}
    By the inclusion-exclusion principle, we get $$\Indic_{\bs D|\bs A; (\bs A/\bs D)_{\le m} = 1} = \sum_{J \subset \I_m} (-1)^{\#J} \Indic_{\bs D\prod_{\bs I\in J}\bs I |\bs A} = \sum_{\bs G = \bs I_1\ldots \bs I_\ell}(-1)^\ell\Indic_{\bs D\bs G|\bs A}.$$

    This sum however includes some terms with $\bs G$ of very large degree (up to $m|\I_m|$). We thus need to cut the sum, using the fact that an inclusion-exclusion always oscillates between both sides of its final value as we add and substract terms: thus, for any cutting point $\ell_0 \in\N,$ we have $$\sum_{\substack{\bs G = \bs I_1\ldots \bs I_\ell \\ \ell\le 2\ell_0-1}}(-1)^\ell\Indic_{\bs D\bs G|\bs A} \le \Indic_{\bs D|\bs A; (\bs A/\bs D)_{\le m} = 1} \le \sum_{\substack{\bs G = \bs I_1\ldots \bs I_\ell \\ \ell\le 2\ell_0}}(-1)^\ell\Indic_{\bs D\bs G|\bs A}.$$

    We thus get $$\sum_{\substack{\bs G = \bs I_1\ldots \bs I_\ell \\ \ell\le 2\ell_0-1}}(-1)^\ell\Pro(\bs D\bs G|\bs A) \le \Pro(\bs D|\bs A; (\bs A/\bs D)_{\le m} = 1) \le \sum_{\substack{\bs G = \bs I_1\ldots \bs I_\ell \\ \ell\le 2\ell_0}}(-1)^\ell\Pro(\bs D\bs G|\bs A),$$ and similarly $$\sum_{\substack{\bs G = \bs I_1\ldots \bs I_\ell \\ \ell\le 2\ell_0-1}}(-1)^\ell\frac 1{\|\bs D\bs G\|} \le \frac1{\|\bs D\|} \Pi_m \le \sum_{\substack{\bs G = \bs I_1\ldots \bs I_\ell \\ \ell\le 2\ell_0}}(-1)^\ell\frac 1{\|\bs D\bs G\|}.$$

    We then get \begin{align*}
        \Pro(\bs D|\bs A; (\bs A/\bs D)_{\le m} = 1) &\le \sum_{\substack{\bs G = \bs I_1\ldots \bs I_\ell \\ \ell\le 2\ell_0}}(-1)^\ell\Pro(\bs D\bs G|\bs A) \\ 
        &\le  \sum_{\substack{\bs G = \bs I_1\ldots \bs I_\ell \\ \ell\le 2\ell_0}}(-1)^\ell\frac 1{\|\bs D\bs G\|} + \sum_{\substack{\bs G = \bs I_1\ldots \bs I_\ell \\ \ell\le 2\ell_0}}\left|\Pro(\bs D\bs G|\bs A)-\frac1{\|\bs D\bs G\|}\right| \\
        &\le \frac1{\|\bs D\|} \Pi_m + \sum_{\substack{\bs G = \bs I_1\ldots \bs I_\ell \\ \ell= 2\ell_0}}\frac 1{\|\bs D\bs G\|} + \sum_{\substack{\bs G = \bs I_1\ldots \bs I_\ell \\ \ell\le 2\ell_0}}\left|\Pro(\bs D\bs G|\bs A)-\frac1{\|\bs D\bs G\|}\right| \\
        &= \frac1{\|\bs D\|} \Pi_m + \frac1{\|\bs D\|}\sum_{\bs G = \bs I_1\ldots \bs I_{2\ell_0}}\frac 1{\|\bs G\|} + \sum_{\substack{\bs G = \bs I_1\ldots \bs I_\ell \\ \ell\le 2\ell_0}}\left|\Pro(\bs D\bs G|\bs A)-\frac1{\|\bs D\bs G\|}\right|.
    \end{align*}

    To conclude, we then need to bound $S_{2\ell_0} := \sum_{\bs G = \bs I_1\ldots \bs I_{2\ell_0}}\frac 1{\|\bs G\|}$ by $\Pi_m$ (where the sum is on products of~$2\ell_0$ distinct irreducibles of $\I_m$) for $\ell_0 := \lceil 2\Sigma_m \rceil.$ To do this, notice we have $$\Sigma_m^{2\ell_0} \ge (2\ell_0)!S_{2\ell_0}$$ (by developing the sum $\Sigma_m$ and keeping only the terms obtained by multiplying $2\ell_0$ distinct terms of the sum), so that 
    \begin{align*}
        S_{2\ell_0} \le \frac{(\Sigma_m)^{2\ell_0}}{(2\ell_0)!} \le \frac{(\ell_0/2)^{2\ell_0}}{(2\ell_0)!} \le (\e/4)^{2\ell_0} &\le (\e/4)^{4\Sigma_m} \\ &= \prod_{\bs I \in \I_m} (\e/4)^{4/\|\bs I\|} \le \prod_{\bs I \in \I_m} 1-1/\|\bs I\| = \Pi_m,
    \end{align*}
    as needed.
\end{proof}

We also need a technical lemma of irrelevance of non-squarefree terms in $\Po$ which will allow us to estimate the distribution of $\tau.$

\begin{lemma}\label{squarefree}
    We have for all $t\in\R,$ $$S_t := \sum_{\bs I\in\Po \text{ irreducible}}\sum_{\nu\ge 2} \frac{(\nu+1)^t}{\|\bs I\|^\nu} < \infty.$$
\end{lemma}
\begin{proof}
    We have \begin{align*}
        \sum_{\bs I\in \Po \text{ irreducible}}\sum_{\nu\ge 2}\frac{(\nu+1)^t}{\|\bs I\|^\nu} &=
        \sum_{i=1}^r\sum_{j=1}^\infty \sum_{I\in \F_{p_i}[X]^u_j \text{ irreducible}} \sum_{\nu\ge 2}\frac{(\nu+1)^t}{p_i^{j\nu}} \\&\le \sum_{i=1}^r\sum_{j=1}^\infty \sum_{\nu\ge 2} \frac{p_i^j}j \frac{(\nu+1)^t}{p_i^{j\nu}} \\&\le r\sum_{\nu=2}^\infty \sum_{j=1}^\infty \frac{(\nu+1)^t}{2^{j(\nu-1)}} \\&\le r\sum_{\nu=2}^\infty \frac{(\nu+1)^t}{2^{\nu-1}-1} <\infty.
    \end{align*}
\end{proof}

\subsection{Proof of Proposition \ref{degree-k-prop}}
\label{large-degree}

We first bound the probability of $\bs A$ having a factor of given degree $\bs k = (k_1,\ldots, k_r)$ without a $\bs X_i$ factor when $\bs A_{\le m}$ is known. We recall the notation $\Po_{\bs k}^X$ for the set of elements of $\Po$ of degree $\bs k$ without any $\bs X_i$ factor.
\begin{lemma}\label{no-Xi}
    Let $m>0,$ $\bs B$ a $m$-friable element of $\Po$ and $\bs k = (k_1,\ldots, k_r)$ with $$\min\bs k \ge \Deg \bs B + m(4\Sigma_m+2).$$ 

    Then for $\bs A$ a random variable in $\Po,$ we have
        $$\Pro\big((\exists\bs D\in\Po_{\bs k}^X,\, \bs D|\bs A) \cap (\bs A_{\le m} = \bs B)\big) \le \frac{4\Pi_m^2 \tau(\bs B)}{\|\bs B\|} + R_\bs B,$$ with an error term $$R_\bs B := \sum_{\bs F|\bs B} \sum_{\substack{\bs E\in \Po_{\bs k-\deg \bs F}^X \\ \bs E_{\le m} = 1}} \sum_{\substack{\bs G=\bs I_1\ldots \bs I_\ell\\ \ell\le 4\Sigma_m + 2\\ \bs I_j\in\I_m \text{ distinct}}} \left|\Pro(\bs{BEG}|\bs A)-\frac1{\|\bs{BEG}\|} \right|.$$
\end{lemma}

\begin{proof}

    We use the fact that when $\bs B$ has few factors, there are few choices of $\bs D_{\le m}$ for $\bs D$ dividing $\bs A,$ as $\bs D_{\le m}$ must then divide $\bs B = \bs A_{\le m}.$ We thus write    
    \begin{align}
        \Pro\big((\exists\bs D\in \Po_{\bs k}^X,\, \bs D|\bs A) \cap (\bs A_{\le m} = \bs B)\big)  
        &\le \sum_{\bs D\in \Po_{\bs k}^X} \Pro\big(\bs D|\bs A \cap (\bs A_{\le m} = \bs B)\big) \notag \\
        &= \sum_{\bs F|\bs B} \sum_{\substack{\bs E \in \Po_{\bs k-\deg \bs F}^X \\ \bs E_{\le m} = 1}} \Pro\big(\bs E|\bs A \cap (\bs A_{\le m} = \bs B)\big) \notag \\
        &\le \sum_{\bs F|\bs B} \sum_{\substack{\bs E\in\Po_{\bs k-\deg \bs F}\\ \bs E_{\le m} = 1}} \frac{2\Pi_m}{\|\bs B\bs E\|} + R_\bs B, \label{step9}
    \end{align}
    where we introduced the change of variable $\bs F = \bs D_{\le m}; \bs E = \bs D_{\ge m}$ and then used Lemma \ref{8.2} to bound $\Pro\big(\bs E|\bs A \cap (\bs A_{\le m} = \bs B)\big) = \Pro\big(\bs B\bs E|\bs A \cap [(\bs A/\bs{BE})_{\le m} = 1]\big)$ which gives the error term $R_\bs B.$ 

    To bound the main term, note that for any $\bs F$ dividing $\bs B,$ if $\Pro_0$ uniform on $\Po_{\bs k-\deg \bs F},$ Lemma \ref{8.2} gives us 
    \begin{equation}\label{step10}
        \sum_{\substack{\bs E \in\Po_{\bs k-\deg \bs F}\\ \bs E_{\le m} = 1}} \frac{1}{\|\bs E\|} = \Pro_0(\bs E_{\le m} = 1) \le 2\Pi_m + \sum_{\substack{\bs G=\bs I_1\ldots \bs I_\ell\\ \ell\le 4\Sigma_m + 2\\ \bs I_j\in\I_m \text{ distinct}}} \left|\Pro_0(\bs G|\bs E)-\frac1{\|\bs G\|} \right|.
    \end{equation}
    
    However, when $\bs G = \bs I_1\ldots \bs I_\ell$ is a product of $\ell\le 4\Sigma_m +2$ distinct factors of $\I_m$, we have $$\max\deg \bs G\le m\ell \le \min\bs k - \Deg \bs B \le \min(\bs k-\deg \bs F)$$ (as we assumed $\min\bs k \ge \Deg \bs B + m(4\Sigma_m+2)$), so that $\Pro_0(\bs G|\bs E) = 1/\|\bs G\|$ for all these $G$ and (\ref{step10}) becomes $$\sum_{\substack{\bs E \in\Po_{\bs k-\deg \bs F}\\ \bs E_{\le m} = 1}} \frac{1}{\|\bs E\|} = \Pro_0(\bs E_{\le m} = 1) \le 2\Pi_m .$$

    Coming back to (\ref{step9}), we then have 
    \begin{align*}
        \Pro\big((\exists\bs D\in \Po_{\bs k}^X,\, \bs D|\bs A) \cap (\bs A_{\le m} = \bs B)\big) 
        &\le \sum_{\bs F|\bs B} \sum_{\substack{\deg \bs E =\bs k-\deg \bs F\\ \bs E_{\le m} = 1}} \frac{2\Pi_m}{\|\bs B\bs E\|} + R_\bs B
        \\&\le \sum_{\bs F|\bs B} \frac{4\Pi_m^2}{\|\bs B\|} + R_\bs B \\&= \frac{4\Pi_m^2 \tau(\bs B)}{\|\bs B\|} + R_\bs B.
    \end{align*}

\end{proof}

We can now move on to the proof of Proposition \ref{degree-k-prop}.

\begin{proof}[Proof of Proposition \ref{degree-k-prop}]

    Our assumptions tell us for all $i, \Pro\left(\bs X_i^{\lceil L\Sigma_m\rceil}|\bs A\right)\le (1-1/L)^{L\Sigma_m} \le \e^{-\Sigma_m}$ so that $$\Pro\left(\exists i, \bs X_i^{\lceil L\Sigma_m\rceil}|\bs A\right) \le r\e^{-\Sigma_m}.$$

    We now consider the case where $\bs X_i^{\lceil L\Sigma_m\rceil}$ does not divide $\bs A$ for any $i.$ Then, if $\bs D$ is a factor of $\bs A$ of~degree~$(k,\ldots, k)$ for some $k\in [v, 2v],$ defining $\bs D_X := \prod_i \bs X_i^{\nu_{\bs X_i}(\bs D)}$ and $\ell = \deg \bs D_X,$ we get $\max \ell < L\Sigma_m$ (as $\bs D_X$ divides $\bs A$) and $\bs D/\bs D_X$ is a factor of $\bs A$ of degree $\bs k'= (k-\ell_1,\ldots, k-\ell_r)$ without any $\bs X_i$ factor. Let $$K = \big\{\bs k =(k_1,\ldots,k_r), \min\bs k> v-L\Sigma_m, \max\bs k \le 2v, \max\bs k -\min\bs k< L\Sigma_m\big\}$$ the set of such possible $\bs k',$ of cardinality $\le (v+L\Sigma_m)(L\Sigma_m)^r.$
    
    When $m>L, \bs k\in K$ and $\bs B$ is a $m$-friable element of $\Ev_m,$ we always have $\Deg \bs B = \Deg \bs B_{\le m} \le m(\Sigma_m - 2),$ so that $$\min\bs k \ge v - L\Sigma_m \ge 5m\Sigma_m \ge \Deg \bs B + m(4\Sigma_m+2)$$ and Lemma \ref{no-Xi} applies. We thus have
    \begin{align}
        \Pro&\left(\left(\exists \bs D\in \bigcup_{v\le k\le 2v} \Po_{(k,\ldots, k)},\, \bs D|\bs A\right) \cap \Ev_m\right) \notag
        \\&\le \sum_{\bs k\in K}\sum_{\substack{\bs B = \bs B_{\le m}\\ \bs B\in\Ev_m}} \Pro\big((\exists\bs D\in \Po_{\bs k}^X,\, \bs D|\bs A) \cap (\bs A_{\le m} = \bs B)\big) + \Pro\left(\exists i, \bs X_i^{\lceil L\Sigma_m\rceil}|\bs A\right) \notag \\
        &\le \#K\cdot \sum_{\substack{\bs B = \bs B_{\le m}\\ \tau(\bs B)\le \e^{(1-1/r)\Sigma_m}}} \frac{4\Pi_m^2 \tau(\bs B)}{\|\bs B\|} + R + r\e^{-\Sigma_m}, \label{step11}
    \end{align}
    with an error term coming from Lemma \ref{no-Xi} $$R := \sum_{\substack{\bs B = \bs B_{\le m}\\ \bs B\in\Ev_m}} R_\bs B = \sum_{\substack{\bs B = \bs B_{\le m}\\ \bs B\in\Ev_m}} \sum_{\bs F|\bs B} \sum_{\bs k\in K} \sum_{\substack{\bs E \in \Po_{\bs k-\deg \bs F}^X \\ \bs E_{\le m} = 1}} \sum_{\substack{\bs G=\bs I_1\ldots \bs I_\ell\\ \ell\le 4\Sigma_m + 2\\ \bs I_j\in\I_m \text{ distinct}}} \left|\Pro(\bs{BEG}|\bs A)-\frac1{\|\bs{BEG}\|} \right|.$$

    To bound the main term, we use Rankin's trick. We have for all $t>0,$
    \begin{align*}
        \sum_{\substack{\bs B = \bs B_{\le m}\\ \tau(\bs B)\le \e^{(1-1/r)\Sigma_m}}} \frac{\tau(\bs B)}{\|\bs B\|} &\le 
        \sum_{\bs B = \bs B_{\le m}} \frac{\tau(\bs B)^{1-t}\e^{(1-1/r)\Sigma_m t}}{\|\bs B\|} \\&= 
        \e^{(1-1/r)\Sigma_m t} \prod_{\bs I\in \I_m} \sum_{\nu\ge 0} \frac{(\nu+1)^{1-t}}{\|\bs I\|^\nu} \\&\le
        \exp\left((1-1/r)\Sigma_m t + \sum_{\bs I\in\I_m}\sum_{\nu\ge 1} \frac{(\nu+1)^{1-t}}{\|\bs I\|^\nu}\right) \\&\le
        \e^{(1-1/r)\Sigma_m t + 2^{1-t}\Sigma_m + S_{1-t}} = O_{r,t}\left(\e^{\left((1-1/r)t + 2^{1-t}\right)\Sigma_m}\right)
    \end{align*} (where $S_{1-t}$ comes from Lemma \ref{squarefree} and depends only on $r$ and $t.$)
    
    Picking the optimal $t := 1-\left(\log\frac{1-1/r}{\log 2}\right)/\log 2,$ we get (as $t$ depends only on $r$) \begin{equation}\label{step12}
        \sum_{\substack{\bs B = \bs B_{\le m}\\ \tau(\bs B)\le \e^{(1-1/r)\Sigma_m}}} \frac{\tau(\bs B)}{\|\bs B\|} =  O_r\left(\e^{\Sigma_m\left[2-2Q\left(\frac{1-1/r}{2\log 2}\right)\right]}\right).
    \end{equation} 

    To bound the error term $R,$ we regroup based on the value of $\bs C = \bs{BEG},$ noting that for each value of $\bs C$ there is at most one choice of $\bs E$ (namely, $\bs E = \bs C_{>m}$) and $\tau(\bs C_{\le m})$ choices for the pair $(\bs B,\bs G)$ that yield~$\bs{BEG} = \bs C.$ Also, given we excluded the $\bs X_i$ factors from the $m$-friable part in Definition \ref{m-friable}, we always have $\bs C\in \Po^X$ when $\bs E = \bs C_{>m}\in \Po^X.$
    \begin{align*}
        R &= \sum_{\substack{\bs B = \bs B_{\le m}\\ \bs B\in\Ev_m}} \sum_{\bs F|\bs B} \sum_{\bs k\in K} \sum_{\substack{\bs E \in \Po_{\bs k-\deg \bs F}^X \\ \bs E_{\le m} = 1}} \sum_{\substack{\bs G= \bs I_1\ldots \bs I_\ell\\ \ell\le 4\Sigma_m + 2\\ \bs I_j\in\I_m \text{ distinct}}} \left|\Pro(\bs{BEG}|\bs A)-\frac1{\|\bs{BEG}\|} \right|
        \\&\le \sum_{\substack{\bs B = \bs B_{\le m}\\ \Deg \bs B\le m(\Sigma_m -2) \\ \tau(\bs B) \le \e^{(1-1/r)\Sigma_m}}} \tau(\bs B) \sum_{\substack{\bs E\in\Po^X\\ \max\deg \bs E \le 2v\\ \bs E_{\le m} = 1}} \sum_{\substack{\bs G = \bs G_{\le m}\\ \Deg \bs G\le m(4\Sigma_m + 2)\\ \tau(\bs G)\le 2^{4\Sigma_m + 2}}} \left|\Pro(\bs{BEG}|\bs A)-\frac1{\|\bs{BEG}\|} \right|
        \\&\le \sum_{\substack{\bs C\in\Po^X \\ \Deg \bs C_{\le m}\le 5m\Sigma_m \\ \Deg \bs C_{>m}\le 2v \\ \tau(\bs C_{\le m})\le 4\e^{(1-1/r+4\log 2)\Sigma_m}}} \tau(\bs C_{\le m})\cdot \e^{(1-1/r)\Sigma_m} \left|\Pro(\bs C|\bs A)-\frac1{\|\bs C\|} \right|
        \\&\le 4\e^{(2-2/r+4\log 2)\Sigma_m} \Delta_\bs A(2v + 5m\Sigma_m).
    \end{align*}

    Combining this with (\ref{step12}) into (\ref{step11}) and using that $\Pi_m\le \e^{-\Sigma_m},$ $\Sigma_m = r(\log m + O(1))$ from Lemma~\ref{pi_m} and $2Q\left(\frac{1-1/r}{2\log 2}\right) = \frac 1r + Q\left(\frac{1-1/r}{\log 2}\right)$, we get 
    \begin{align*}
        \Pro&\left(\left(\exists \bs D\in \bigcup_{v\le k\le 2v} \Po_{(k,\ldots, k)},\, \bs D|\bs A\right) \cap \Ev_m\right) \\
        &\le (v+L\Sigma_m)(L\Sigma_m)^r \cdot 4\Pi_m^2\cdot O_r\left(\e^{\Sigma_m\left[2-2Q\left(\frac{1-1/r}{2\log 2}\right)\right]}\right) + 4\e^{(2-2/r+4\log 2)\Sigma_m} \Delta_\bs A(2v + 5m\Sigma_m) + r\e^{-\Sigma_m}
        \\&\le O_r\left(v(L\Sigma_m)^r m^{-2rQ\left(\frac{1-1/r}{2\log 2}\right)} + m^{5r}\Delta_\bs A(2v + 5m\Sigma_m) + m^{-r}\right) \\&=  O_r\left(\frac vm(L\log m)^r m^{-rQ\left(\frac{1-1/r}{\log 2}\right)} + m^{5r}\Delta_\bs A(2v + 5m\Sigma_m)\right).
    \end{align*}

    When $m\le L,$ we have $\Pro\left(\left(\exists \bs D\in \bigcup_{n_1\le k\le n_2} \Po_{(k,\ldots, k)},\, \bs D|\bs A\right) \cap \Ev_m\right)\le 1\le L^rm^{-r}$ so the conclusion still holds.
\end{proof}

\subsection{Ruling out large $\bs A_{\le m}$}
\label{many-factors}

We now want to bound the probability of $\overline{\Ev_m}.$ We start by bounding the probability of the condition $\Deg \bs A_{\le m} \ge m(\Sigma_m -2).$
\begin{lemma}\label{deg-le-m}
    Let $C>0$ a constant. For all $u,$ we get $$\Pro\left(\Deg \bs A_{\le m} \ge um\right) \le O_{r,C}(m^{2r}\e^{-Cu}) + \Delta_\bs A\big((u+1)m\big).$$
\end{lemma}
Notably with $C=3$ and $u = \Sigma_m -2,$ we get $$\Pro\left(\Deg \bs A_{\le m} \ge m(\Sigma_m -2)\right) \le O_r(m^{2r}\e^{-3\Sigma_m}) + \Delta_\bs A\big(m(\Sigma_m-1)\big) \le  O_r(m^{-r}) + \Delta_\bs A(m\Sigma_m)$$ (as $\Sigma_m = r(\log m + O(1))$ by Lemma \ref{pi_m} and $\Delta_\bs A$ is increasing).

The lemma above is a generalisation of \cite[lemma 9.1]{koukoulo} working with $r$ primes instead of one. The proof, given below, follows the same steps as the reference above.
\begin{proof}
    If $\Deg \bs A_{\le m} \ge um,$ then consider $\bs D$ a factor of $\bs A_{\le m}$ of total degree $\ge um,$ but of minimal total degree among these. As all irreducible factors of $\bs A_{\le m}$ have total degree $\le m,$ this is also true for $\bs D.$ Thus, if $\bs D$ had total degree $\ge (u+1)m,$ we could remove from it an irreducible factor and keep a factor of degree~$\ge um,$ contradicting the minimality of $\bs D.$ Thus, when $\Deg \bs A_{\le m} \ge um,$ $\bs A$ always has a $m$-friable factor $\bs D$ of total degree between $um$ and $(u+1)m.$

    We can then bound \begin{align*}
        \Pro\left(\Deg \bs A_{\le m} \ge um\right) &\le \sum_{\substack{um\le \Deg \bs D\le (u+1)m\\ \bs D=\bs D_{\le m}}} \Pro(\bs D|\bs A)
        \\&\le \sum_{\substack{um\le \Deg \bs D\\ \bs D= \bs D_{\le m}}} \frac1{\|\bs D\|} + \sum_{\substack{\Deg \bs D\le (u+1)m\\ \bs D=\bs D_{\le m}}} \left|\Pro(\bs D|\bs A)-\frac1{\|\bs D\|}\right|
        \\ &\le \sum_{\substack{um\le \Deg \bs D\\ \bs D= \bs D_{\le m}}} \frac1{\|\bs D\|} + \Delta_\bs A\big((u+1)m\big)
    \end{align*}
    (note here the $\bs D=\bs D_{\le m}$ condition implies the $\bs X_i$ do not divide $\bs D$ as we excluded these factors from the $m$-friable part in Definition \ref{m-friable}, which allows to bound the error term by a $\Delta_\bs A$ term.)
    
    We then use Rankin's trick to bound the main term. Adjusting the implicit constant of the big-Oh term in the final result, we can suppose $m\ge 3C$ and get
    \begin{align*}
        \sum_{\substack{um\le \Deg \bs D\\ \bs D=\bs D_{\le m}}} \frac1{\|\bs D\|}
        &\le \sum_{\bs D=\bs D_{\le m}} \frac{\e^{C\Deg \bs D/m-Cu}}{\|\bs D\|}
        \\ &= \e^{-Cu}\prod_{\bs I\in \I_m}\left(1-\frac{\e^{C\Deg \bs I/m}}{\|\bs I\|} \right)^{-1},
    \end{align*}
    where we indeed have $\frac{\e^{C\Deg \bs I/m}}{\|\bs I\|} \le (\e^{C/m}/2)^{\Deg \bs I} \le (\e^{1/3}/2)^{\Deg \bs I} <1$ for all $\bs I \in \I_m$ which justifies the convergence of the series and the last equality.

     Moreover, for all $x \in [0,\e^{1/3}/2],$ we have $-\log (1-x)\le 2x$ so that 
     \begin{align*}
         \log \prod_{\bs I\in \I_m}\left(1-\frac{\e^{C\Deg \bs I/m}}{\|\bs I\|} \right)^{-1} &\le 2\sum_{\bs I\in \I_m} \frac{\e^{C\Deg \bs I/m}}{\|\bs I\|} \\&\le 2\sum_{\bs I\in \I_m} \frac{1+ (\e^C-1)\Deg \bs I/m}{\|\bs I\|} \\&= 2\sum_{i=1}^r \sum_{j=1}^m \frac{\#\{I \in \F_{p_i}[X] \text{ irreducible of degree }j, I\ne X\}}{p_i^j}(1+ (\e^C-1)j/m) \\&\le 2r\sum_{j=1}^m \frac{1+(\e^C-1)j/m}j \\&\le 2r(\log m +\e^C), 
     \end{align*}
     so that $$\sum_{\substack{um\le \Deg \bs D\le (u+1)m\\ \bs D=\bs D_{\le m}}} \frac1{\|\bs D\|} \le \e^{-Cu}\prod_{\bs I\in \I_m}\left(1-\frac{\e^{C\Deg \bs I/m}}{\|\bs I\|} \right)^{-1} \le \e^{2r\e^C}\e^{-Cu}m^{2r}$$ as we want.
\end{proof}

We now show that the number of distinct irreducible factors of $\bs A_{\le m}$ is unlikely to be very large. This will be useful to bound the error term when trying to estimate the distribution of $\tau(\bs A_{\le m}).$

\begin{lemma}\label{omega}
     We have when $u>1,$ $$\Pro\big(\omega(\bs A_{\le m}) \ge u\Sigma_m\big) \le O_{u,r}\left(\e^{(1-Q(u))\Sigma_m}\right) + \Delta_\bs A(mu\Sigma_m),$$ where $\omega(\bs B)$ is the number of distinct irreducible factors of $\bs B$ for $\bs B\in\Po.$
\end{lemma}
\begin{proof}
    As in the proof of Lemma \ref{deg-le-m}, we start by noting that if $\omega(\bs A_{\le m}) \ge u\Sigma_m,$ then $\bs A_{\le m}$ has a factor $\bs B$ which is a product of $\lfloor u\Sigma_m\rfloor$ distinct irreducible factors. This factor then has a total degree $\Deg \bs B\le mu\Sigma_m,$ and verifies $\omega(\bs B) = \lfloor u\Sigma_m\rfloor \ge u\Sigma_m-1,$ so that
    \begin{align*}
        \Pro\big(\omega(\bs A_{\le m}) \ge u\Sigma_m\big) &\le \sum_{\substack{\bs B=\bs B_{\le m}\\ \omega(\bs B)\ge u\Sigma_m-1\\ \Deg \bs B\le mu\Sigma_m}} \Pro(\bs B|\bs A) 
        \\&\le \sum_{\substack{\bs B= \bs B_{\le m}\\ \omega(\bs B)\ge u\Sigma_m-1}} \frac1{\|\bs B\|} + \sum_{\substack{\bs B= \bs B_{\le m}\\ \Deg \bs B\le mu\Sigma_m}} \left|\Pro(\bs B|\bs A)-\frac1{\|\bs B\|}\right|
        \\&\le \sum_{\substack{\bs B= \bs B_{\le m}\\ \omega(\bs B)\ge u\Sigma_m-1}} \frac1{\|\bs B\|} + \Delta_\bs A(mu\Sigma_m).
    \end{align*}
    
    To bound the main term, we use Rankin's trick. For $t>0,$ we have
    \begin{align*}
        \sum_{\substack{\bs B= \bs B_{\le m}\\ \omega(\bs B)\ge u\Sigma_m-1}} \frac1{\|\bs B\|} &\le \sum_{\bs B= \bs B_{\le m}} \frac{\e^{t(\omega(\bs B) - u\Sigma_m+1)}}{\|\bs B\|} \\&= \e^{-tu\Sigma_m+t}\prod_{\bs I\in\I_m}\left(1+ \sum_{\nu\ge 1}\frac{\e^t}{\|\bs I\|^\nu} \right) \\&\le \e^{-tu\Sigma_m + \e^t\Sigma_m + S_t+t}
    \end{align*}
    (where $S_t$ comes from Lemma \ref{squarefree} and we bounded $\e^t\le (\nu+1)^t$ for $\nu\ge 2.$)

    Taking the optimal $t := \log u >0$ (when $u>1$), we get the desired bound $$\e^{(1-Q(u))\Sigma_m + S_{\log u}+\log u} = O_{r,u}\left(\e^{(1-Q(u))\Sigma_m}\right).$$
\end{proof}

We now have everything we need to bound $\Pro\left(\overline{\Ev_m}\right).$
\begin{prop}\label{lemma-Ev}
    We have when $r\ge 4, m>0$ and $\bs A$ is a random variable in $\Po,$ $$\Pro\left(\overline{\Ev_m}\right) = O_r\left(m^{-rQ\left(\frac{1-1/r}{\log 2}\right)} + m^{5r}\Delta_\bs A(5m\Sigma_m)\right).$$
\end{prop}
\begin{proof}

    We have \begin{multline}\label{eq2}\Pro\left(\overline{\Ev_m}\right) \le \Pro\big(\Deg \bs A_{\le m}\ge m(\Sigma_m-2)\big) + \Pro\big(\omega(\bs A_{\le m})\ge 3\Sigma_m\big) \\+ \Pro\big(\tau(\bs A_{\le m})\ge \e^{(1-1/r)\Sigma_m}; \Deg \bs A_{\le m}\le m(\Sigma_m-2),\omega(\bs A_{\le m})\le 3\Sigma_m\big)\end{multline}
    with by Lemma \ref{deg-le-m}, $$\Pro\big(\Deg \bs A_{\le m} \ge m(\Sigma_m-2)\big) \le O_r(m^{-r}) + \Delta_\bs A(m\Sigma_m),$$ and by Lemma \ref{omega}, $$\Pro\big(\omega(\bs A_{\le m})\ge 3\Sigma_m\big) \le O_r(\e^{(1-Q(3))\Sigma_m}) + \Delta_\bs A(3m\Sigma_m).$$

    We can now estimate the main term $\Pro\big(\tau(\bs A_{\le m})\ge \e^{(1-1/r)\Sigma_m}; \Deg \bs A_{\le m}\le m(\Sigma_m-2),\omega(\bs A_{\le m})\le 3\Sigma_m\big)$ using our sieve lemma \ref{8.2}:
    \begin{align*}
        \Pro\big(&\tau(\bs A_{\le m})\ge \e^{(1-1/r)\Sigma_m}; \Deg \bs A_{\le m}\le m(\Sigma_m-2),\omega(\bs A_{\le m})\le 3\Sigma_m\big)
        \\&= \sum_{\substack{\bs B= \bs B_{\le m}\\ \tau(\bs B)\ge \e^{(1-1/r)\Sigma_m}\\ \Deg \bs B\le m(\Sigma_m-2)\\ \omega(\bs B)\le 3\Sigma_m}} \Pro(\bs A_{\le m}=\bs B)
        \\&\le \sum_{\substack{\bs B= \bs B_{\le m}\\ \tau(\bs B)\ge \e^{(1-1/r)\Sigma_m}}} 2\frac{\Pi_m}{\|\bs B\|} + \sum_{\substack{\bs B= \bs B_{\le m}\\\Deg \bs B\le m(\Sigma_m-2)\\\omega(\bs B)\le 3\Sigma_m}} \sum_{\substack{\bs G = \bs I_1\ldots \bs I_\ell\\I_j\in \I_m \text{ distinct}\\\ell\le 4\Sigma_m+2}} \left|\Pro(\bs B\bs G|\bs A)-\frac1{\|\bs B\bs G\|}\right|.
    \end{align*}

    To bound the main term, we use Rankin's trick again. We have for all $t>0,$
    \begin{align*}
        \sum_{\substack{\bs B = \bs B_{\le m}\\ \tau(\bs B)\ge \e^{(1-1/r)\Sigma_m}}} \frac{1}{\|\bs B\|} &\le 
        \sum_{\bs B = \bs B_{\le m}} \frac{\tau(\bs B)^{t}\e^{-t(1-1/r)\Sigma_m}}{\|\bs B\|} \\&= 
        \e^{-t(1-1/r)\Sigma_m} \prod_{\bs I\in \I_m} \sum_{\nu\ge 0} \frac{(\nu+1)^{t}}{\|\bs I\|^\nu} \\&\le
        \exp\left(-t(1-1/r)\Sigma_m + \sum_{\bs I\in\I_m}\sum_{\nu\ge 1} \frac{(\nu+1)^{t}}{\|\bs I\|^\nu}\right) \\&\le
        \e^{-t(1-1/r)\Sigma_m + 2^{t}\Sigma_m + S_{t}} = O_{r,t}\left(\e^{\left((1-1/r)t + 2^{t}\right)\Sigma_m}\right),
    \end{align*}
    where $S_t$ comes from Lemma \ref{squarefree}. Taking the optimal value $t =\frac{\log\frac{1-1/r}{\log 2}}{\log 2} >0$ (as $r\ge 4 > \frac{1}{1-\log 2}$), we get $$\sum_{\substack{\bs B = \bs B_{\le m}\\ \tau(\bs B)\ge \e^{(1-1/r)\Sigma_m}}} 2\frac{\Pi_m}{\|\bs B\|} =O_r\left(2\Pi_m\e^{\left[1-Q\left(\frac{1-1/r}{\log 2}\right)\right]\Sigma_m}\right) = O_r\left(\e^{-Q\left(\frac{1-1/r}{\log 2}\right)\Sigma_m}\right).$$

    To bound the error term, we regroup based on the value of $\bs C=\bs B\bs G.$ We then have
    \begin{align*}
        \sum_{\substack{\bs B= \bs B_{\le m}\\\Deg \bs B\le m(\Sigma_m-2)\\\omega(\bs B)\le 3\Sigma_m}} \sum_{\substack{\bs G = \bs I_1\ldots \bs I_\ell\\I_j\in \I_m \text{ distinct}\\\ell\le 4\Sigma_m+2}} \left|\Pro(\bs B\bs G|\bs A)-\frac1{\|\bs B\bs G\|}\right| &\le \sum_{\substack{\bs C = \bs C_{\le m}\\ \Deg \bs C\le 5m\Sigma_m \\ \omega(\bs C)\le 7\Sigma_m +2}} \#\{\bs G|\bs C \text{ square-free}\} \left|\Pro(\bs C|\bs A)-\frac1{\|\bs C\|}\right|
        \\&\le 2^{7\Sigma_m+2}\Delta_\bs A(5m\Sigma_m),
    \end{align*}
    remembering $\bs C = \bs C_{\le m}$ implies $\bs C$ has no $\bs X_i$ factor (given Definition \ref{m-friable}) and as such the sum on $\bs C$ can be bounded by a $\Delta_\bs A$ term.

    Combining this and the results from Lemmas \ref{deg-le-m} and \ref{omega} in (\ref{eq2}), we get
    \begin{align*}
        \Pro\left(\overline{\Ev_m}\right) &\le \begin{aligned}[t]
             O_r(m^{-r}) + \Delta_\bs A(m\Sigma_m) + O_r&\left(e^{(1-Q(3))\Sigma_m}\right) + \Delta_\bs A(3m\Sigma_m) \\&+ O_r\left(\e^{-Q\left(\frac{1-1/r}{\log 2}\right)\Sigma_m}\right) + 2^{7\Sigma_m+2}\Delta_\bs A(5m\Sigma_m)
        \end{aligned} \\&= O_r\left(m^{-rQ\left(\frac{1-1/r}{\log 2}\right)} + m^{5r}\Delta_\bs A(5m\Sigma_m)\right),
    \end{align*}
    as $\Sigma_m = r(\log m + O(1))$ by Lemma \ref{pi_m}, $Q(3)-1\ge Q(1/\log 2) \ge Q\left(\frac{1-1/r}{\log 2}\right)$ and~$\Delta_\bs A(m\Sigma_m)\le~ \Delta_\bs A(3m\Sigma_m)\le \Delta_\bs A(5m\Sigma_m).$
\end{proof}

\subsection{Proof of Proposition \ref{main-mod-pi}}

We can now combine Propositions \ref{degree-k-prop} and \ref{lemma-Ev} to prove Proposition \ref{main-mod-pi}.

\begin{proof}[Proof of Proposition \ref{main-mod-pi}]
    For $v>0,$ fix $m = m(v) := v/(\log v)^5.$ When $n_1$ is large enough compared to $r$ (which we can assume adjusting the implicit constant of the final big-Oh term), we then have $5m\Sigma_m \le~v/(\log v)^3$ and $v\ge 6m\Sigma_m$ as long as $v\ge n_1/2$ by Lemma \ref{pi_m}. Then, Propositions \ref{degree-k-prop} and \ref{lemma-Ev} give for these $v,$
    \begin{align*}
        \Pro&\left(\exists \bs D\in \bigcup_{v\le k\le 2v} \Po_{(k,\ldots, k)},\, \bs D|\bs A\right)
        \\&\le \Pro\left(\left(\exists \bs D\in \bigcup_{v\le k\le 2v} \Po_{(k,\ldots, k)},\, \bs D|\bs A\right) \cap \Ev_{m(v)}\right) + \Pro\left(\overline{\Ev_{m(v)}}\right)
        \\&= O_r\left(\frac vm(L\log m)^r m^{-rQ\left(\frac{1-1/r}{\log 2}\right)} + m^{5r}\Delta_\bs A(2v + 5m\Sigma_m)\right) + O_r\left(m^{-rQ\left(\frac{1-1/r}{\log 2}\right)} + m^{5r}\Delta_\bs A(5m\Sigma_m)\right)
        \\&= O_r\left(L^r v^{-rQ\left(\frac{1-1/r}{\log 2}\right) + o(1)} + v^{5r}\Delta_\bs A\left(2v + v/(\log v)^3\right)\right).
    \end{align*}

    Applying for $v_j = n_2/2^j$ and summing on $1\le j\le 1+ \frac{\log \frac{n_1}{n_2}}{\log 2}$ (so that the intervals $[v_j, 2v_j]$ cover $[n_1,n_2]$ and we always have $n_1/2\le v_j\le n_2/2$), we get
    $$\Pro\left(\exists \bs D\in \bigcup_{n_1\le k\le n_2} \Po_{(k,\ldots, k)},\, \bs D|\bs A\right) = O_r\left(L^r n_1^{-rQ\left(\frac{1-1/r}{\log 2}\right) + o(1)} + n_2^{5r}\Delta_\bs A(n_2 + n_2/(\log n_2)^3)\right).$$
\end{proof}

\section{Proofs of the main results}
We now come back to the settings of the introduction, with $A \in \Z[X]$ a random monic polynomial of fixed degree $n$ and with coefficients drawn independently according to the $\mu_j.$

\subsection{Factors of small degree}\label{small-degree}
We prove in this section Lemma \ref{small-cyclo}. We use the following result. 
\begin{lemma}[Kolmogorov-Rogozin, \cite{rogozin}]\label{rogozin}
    Let $Q(Y;\delta) := \sup_x \Pro(|Y-x|\le \delta).$ Then if $X_1, \ldots, X_k$ are independent and $S := \sum_{i=1}^k X_i,$ we have for $L \ge \ell,$
    $$Q(S; L) \le K\frac{L}{\ell} \left(\sum_{i=1}^k \big(1-Q(X_i;\ell)\big)\right)^{-1/2}$$ for some absolute constant $K.$
\end{lemma}

\begin{proof}[Proof of Lemma \ref{small-cyclo}]
    As $\Phi_d$ divides $X^d - 1,$ we have $$A \equiv \sum_{i=0}^{d-1} a'_iX^i \mod \Phi_d,$$ where $$a'_i=\sum_{j \equiv i \Mod d} a_j.$$ 
    
    Now, as $\Phi_d$ has degree $\varphi(d),$ for any choice of the $a'_i, i\ge \varphi(d),$ there is only one choice of the $a'_i, i<\varphi(d)$ such that $\Phi_d|A.$ As the $a'_i$ are independent, we thus get $$\Pro(\Phi_d|A)\le \prod_{i=0}^{\varphi(d)-1} \sup_x \Pro(a'_i=x).$$

    Now, as the $a_j$ take values in $\Z$ for all $j,$ we get $Q(a_j,1/2) = \sup_x \Pro(a_j=x) \le 1-c.$ Lemma \ref{rogozin} then gives for all $i,$ as $a'_i$ is a sum of $n/d + O(1)$ variables $a_j,$ $$\sup_x \Pro(a'_i=x) = Q(a'_i, 1/2) \le\sqrt{\frac{Kd}{cn}}$$ for a constant $K,$ so that $$\Pro(\Phi_d|A) \le \prod_{i=0}^{\varphi(d)-1} \sup_x \Pro(a'_i=x) \le \left(\frac{Kd}{cn}\right)^{\varphi(d)/2}.$$
\end{proof}

From this and Proposition \ref{small-degree-lemma}, we can deduce the following.

\begin{prop}\label{small-degree-full}
    Assume all the $\mu_j$ have support in $[-H,H]$ for some $H > 0,$ and do not take any value with probability higher than $1-c$ for some $c>0.$ Fix then $C > 0$ and let $U_C$ the finite subset of $\C$ consisting of $0$ and all roots of unity $\omega$ such that $[\Q(\omega):\Q] < C.$ Then for $n\ge (\log H)^3,$
    \begin{equation}
        \Pro\big(A \text{ has a factor of degree}\le n^{1/10}\big) = \Pro\big(\exists \omega\in U_C : A(\omega) = 0\big) + O_{C,c}(n^{-C/2}).
    \end{equation}
\end{prop}

\begin{proof}
    We treat non-cyclotomic, small cyclotomic and large cyclotomic factors separately. Note that $U_C$ contain exactly the complex roots of $X$ and cyclotomic factors of degree $<C,$ so that 
    \begin{multline*}
        0 \le \Pro\big(A \text{ has a factor of degree}\le n^{1/10}\big) - \Pro\big(\exists \omega\in U_C : A(\omega) = 0\big) \\ \le \Pro\big(A \text{ has a non-$X,$ non-cyclotomic factor of degree} \le n^{1/10}\big) + \sum_{C \le \varphi(d) \le n^{1/10}} \Pro(\Phi_d|A).
    \end{multline*}

    Using Proposition \ref{small-degree-lemma} to bound $\Pro\big(A \text{ has a non-$X,$ non-cyclotomic factor of degree} \le n^{1/10}\big),$ it is thus sufficient to prove $$\sum_{C \le \varphi(d) \le n^{1/10}} = O_{C,c}(n^{-C/2}).$$ 

    But using Lemma \ref{small-cyclo} and the fact probabilities are always smaller than one, we have
    
    \begin{align*}
        \sum_{C \le \varphi(d)\le n^{1/10}} \Pro(\Phi_d|A) &\le \sum_{C \le \varphi(d)\le n^{1/10}} \left(\min\left\{1;\frac{Kd}{cn}\right\}\right)^{\varphi(d)/2} \\ 
        &\le \sum_{C \le \varphi(d) < 2C} \left(\frac{Kd}{cn}\right)^{C/2} + \sum_{2C \le \varphi(d)\le n^{1/10}} \left(\frac{Kd}{cn}\right)^{C}
        &= O_{C,c}(n^{-C/2}).
    \end{align*}
\end{proof}

\subsection{The case of a segment}\label{uniform-final}
We now fix all $\mu_j$ equal to $\mu$ uniform over a segment $\llbracket a, a+N-1\rrbracket$ of length $N$ in order to prove Theorems~\ref{main-uniform} and \ref{uniform-case}. We first prove $\widehat\mu$ verifies (\ref{condition-hatmu}) in this case.

\begin{prop}\label{unif-Q}
    Let $\mu$ an uniform measure over a segment $\llbracket a, a+N-1\rrbracket$ of $\Z$ of length $N,$ and $Q$ a positive integer. Then for any $\theta_0\in\R/\Z,$ we have $$\sum_{k\in \Z/Q\Z} \big|\widehat\mu(k/Q+\theta_0)\big| \le 1 + \frac{Q\cdot\big(\log(Q-1) + 2\big)}{N}.$$
\end{prop}
\begin{proof}
    Here, for $\theta \not\equiv 0 \mod 1,$ we have $$\widehat\mu(\theta)=\frac 1N\sum_{j=a}^{a+N-1} \e^{2\ii\pi j\theta} = \frac1N \e^{2\ii\pi a\theta}\frac{1-\e^{2\ii\pi N\theta}}{1-\e^{2\ii\pi\theta}}, $$
    so that $$|\widehat\mu(\theta)| \le \frac 2{N|1-\e^{2\ii\pi \theta}|} =\frac 1{N|\sin\pi\theta|}.$$
    We also always have $\big|\widehat\mu(\theta)\big| \le\frac 1N \sum_{j=a}^{a+N-1} \big|\e^{2\ii\pi j\theta}\big| = 1$ for all $\theta.$

    Now, note that the sum we want to bound does not change when $\theta_0$ is translated by a multiple of $1/Q.$ Thus, we can suppose without loss of generality that $\theta_0\in ]-1/2Q, 1/2Q],$ so that when $k\in [-Q/2, Q/2]$ with $|k|\ge 1,$ we have $$\big|\sin\pi(k/Q+\theta_0)\big|\ge \min_{(k-1/2)/Q\le\theta\le(k+1/2)/Q} |\sin\pi\theta| = \sin\pi\frac{|k|-1/2}Q \ge \frac{2|k|-1}{Q}$$ (using $\sin x \ge \frac2\pi x$ in $[0,\pi/2]$).

    Then, rewriting the sum over $\Z/Q\Z$ as a sum over integers in $]-Q/2,Q/2],$ we get
    \begin{align*}
        \sum_{k\in \Z/Q\Z} \big|\widehat\mu(k/Q+\theta_0)\big| &\le \big|\widehat\mu(\theta_0)\big| + \sum_{1\le|k|\le Q/2} \frac{1}{N\sin\pi(k/Q+\theta_0)} \\&\le 1 + \sum_{1\le |k|\le Q/2}\frac{Q}{N(2|k|-1)}
        \\& = 1 + 2\frac Q{2N} \sum_{k=1}^{\lfloor Q/2\rfloor}\frac 1{k-1/2}
        \\&\le 1 + \frac QN \left(2 +\int_{1/2}^{Q/2-1/2}\frac 1t\,\de t\right) = 1 + \frac{Q\cdot\big(\log(Q-1) +2\big)}{N}.
    \end{align*}
\end{proof}

We can now prove Theorem \ref{main-uniform}.

\begin{proof}[Proof of Theorem \ref{main-uniform}]
    We start by proving the second statement of the theorem. Here, Proposition \ref{small-degree-full} applies with $c = 1/2$ since the $\mu_j$ are uniform over a segment of length $N\ge 2.$ Thus, it is sufficient to show $$\Pro\big(A \text{ has a factor of degree } \in [n^{1/10},\delta n]\big) = O_C(n^{-C/2})$$
    for some $\delta$ depending only on $C.$

    Fix $r_1\ge 4$ such that $r_1Q\left(\frac{1-1/r_1}{\log 2}\right) > 5C$ (this is possible since $rQ\left(\frac{1-1/r}{\log 2}\right) \asymp r \to \infty$ when $r\to\infty$), and $P = p_1\ldots p_{r_1}$ the product of the $r_1$ smallest prime numbers. 
    
    Now, let $\mu$ the uniform measure on $\llbracket a, a+N-1\rrbracket$ (which depends on $N$ and $a$), and note that for all $\theta\not\equiv 0\mod 1,$ we have $\max_{N\ge 2} |\widehat\mu(\theta)|<1$ (this is a $\max$ because $|\widehat\mu(\theta)|$ does not depend on $a$ and tends to $0$ when $N\to\infty$), so that $|\widehat\mu(\theta)|^s\to 0$ when $s\to\infty,$ uniformly in $N$ and $a.$ Thus, for any fixed $Q, R=P/Q, \ell\in\Z/R\Z,$ for $s$ large enough (depending only on $Q,R,\ell$), we have $$\sum_{k\in \Z/Q\Z} \big|\widehat{\mu_j}(k/Q+\ell/R)\big|^s\le \left(1-\frac1{4\log P}\right)\sqrt Q.$$ 

    As $P$ has a finite number of divisors, we can then fix $s$ (depending only on $P,$ which depends only on $C$) which verifies the condition above for all $Q$ dividing $P$ and $\ell\in\Z/R\Z.$ Then with $\delta := 1/2s,$ Theorem \ref{general-case} with $n_1 = n^{1/10}$ gives us for $n\ge P^4,$ $$\Pro\big(A \text{ has a factor of degree} \in [n^{1/10}, \delta n]\big) = O_{r_1,s}\left(n^{-\frac{1}{10} r_1Q\left(\frac{1-1/r_1}{\log 2}\right) + o(1)}\right) = O_{C}(n^{-C/2})$$ by definition of $r_1,$ as desired. As $P$ depends only on $C$ and $\eps,$ and probabilities are always bounded by $1,$ we can adjust the implicit constant of the big-Oh term to extend it to the values of $n\le P^4$ as well.

    We now prove the first statement. This time, our starting point is the second statement (which we just proved), which tells us
    \begin{equation}\label{lemma4}
        \Pro\big(A \text{ has a factor of degree}\le \delta n\big) = \Pro\big(\exists \omega\in U_C : A(\omega) = 0\big) + O_{C}(n^{-C/2})
    \end{equation} for some $\delta$ depending only on $C.$

    Now, fix $r\ge 4$ such that $rQ\left(\frac{1-1/r}{\log 2}\right) > C/2$ and $P = p_1\ldots p_r$ the product of the $r$ smallest primes, and \begin{equation}\label{bound-N}
        N_0 \ge f(P) := \frac{P\cdot\big(\log (P-1) + 2\big)}{0.99\sqrt P - 1}
    \end{equation} (the constant $0.99$ here is arbitrary, and $N_0$ depends only on $C$ as $P$ does. As a matter of fact, when $C = 1,$ these constraints are satisfied with $r=12,$ $P = 2\cdot 3\cdot\ldots\cdot 37\simeq 7.4\cdot 10^{12}$ and $N_0 = 10^8.$)

    For $N\ge N_0$ and $2\le Q\le P,$ as $f$ is increasing on $[2,+\infty[,$ we then have $N\ge N_0\ge f(P)\ge f(Q)$ and then by Proposition \ref{unif-Q}, for all $\theta_0,$ $$\sum_{k\in \Z/Q\Z} \big|\widehat\mu(k/Q+\theta_0)\big| \le 1 + \frac{Q\cdot\big(\log(Q-1)+2\big)}{N} \le 1 + \frac{Q\cdot\big(\log(Q-1)+2\big)}{f(Q)} = 0.99\sqrt Q,$$ where $\mu$ uniform over $\llbracket a, a+N-1\rrbracket.$

    Thus, Theorem \ref{general-case} applies with $s=1$ for $n\ge\max\{P^4, \e^{100}\}$ and gives us $$\Pro\big(A \text{ has a factor of degree} \in [\delta n, n/2]\big) = O_{r,s}\left((\delta n)^{-rQ\left(\frac{1-1/r}{\log 2}\right) + o(1)}\right) = O_{C}(n^{-C/2}),$$ a result which can again be extended to smaller values of $n$ adjusting the implicit constant of the big-Oh term. Combining with (\ref{lemma4}) concludes the proof.
\end{proof}




\section{Acknowledgments}
This work was the result of my research for my master's thesis, under supervision of R\'egis de la Bret\`eche (Universit\'e Paris Cit\'e). I would like to thank him for presenting me the topic and guiding me throughout the writing of this paper, as well as his useful review and commentaries on it. I would also like to thank the anonymous reviewer for their valuable remarks on the paper.


\end{document}